\documentclass[a4paper,12pt]{article}
\usepackage{a4wide}
\usepackage{amsmath}
\usepackage{amssymb}
\usepackage{amsthm}
\usepackage{latexsym}
\usepackage{graphicx}
\usepackage[english]{babel}
\usepackage{makeidx}
\usepackage{mathtools}

\newtheorem{obs} [subsection]{Remark}
\newtheorem{exm} [subsection]{Example}

\newtheorem{prop}[subsection]{Proposition}

\newtheorem{teor}[subsection]{Theorem}

\newtheorem{cor} [subsection]{Corollary}

\newcommand{\pa}{p_{\mathbf a}}
\newcommand{\fa}{f_{\mathbf a}}

\newcommand{\za}{\zeta_{\mathbf a}}

\newcommand{\Res}{Res}

\def\lcm{\operatorname{lcm}}

\def\gr{\operatorname{gr}}
\numberwithin{equation}{section}

\begin{document}
\selectlanguage{english}
\frenchspacing

\large
\begin{center}
\textbf{On a Zeta-Barnes type function associated to graded modules}

Mircea Cimpoea\c s
\end{center}
\normalsize

\begin{abstract}
Let $K$ be a field and let $S=\bigoplus_{n\geq 0} S_n$ be a positively graded $K$-algebra.
Given $M=\bigoplus_{n\geq 0} M_n$, a finitely generated graded $S$-module, and $w>0$,
we introduce the function $\zeta_M(z,w):= \sum_{n=0}^{\infty}\frac{H(M,n)}{(n+w)^z}$,  where $H(M,n):=\dim_K M_n$, $n\geq 0$,
is the Hilbert function of $M$, and we study the relations between the algebraic properties of $M$ and the
analytic properties of $\zeta_M(z,w)$. In particular, in the standard graded case, we prove that the multiplicity of $M$, 
$e(M)=(m-1)!\lim_{w\searrow 0}\Res_{z=m}\zeta_M(z,w)$.


\end{abstract}

\section*{Introduction}

Let $K$ be a field and let $S$ be a positively graded $K$-algebra. 
Let $M$ be a finitely generated $S$-module of dimension $m\geq 0$. 
Given a real number $w>0$, we consider the function
$$\zeta_M(z,w):=\sum_{n=0}^{\infty}\frac{H(M,n)}{(n+w)^z},$$
where $H(M,n):=\dim_K M_n$, $n\geq 0$, is the Hilbert function of $M$. According to a Theorem of Serre, see for instance \cite[Theorem 4.4.3]{bh}, there exists a positive integer $D$ such that
$$H(M,n)=d_{M,m-1}(n)n^{m} + \cdots + d_{M,1}(n)n + d_{M,0}(n),\;(\forall)n\gg 0,$$
where $d_{M,j}(n+D)=d_{M,j}(n)$, $(\forall)n\geq 0$. We introduce the zeta-Barnes \cite{barnes} type function
$ \zeta_M(z,w):=\sum_{n\geq 0}\frac{H(M,n)}{(n+w)^z},\; z\in\mathbb C$,
and we study its properties. In Theorem $1.1$ we show that $\zeta_M(z,w)$ can be written as a linear combination of
Hurwitz-Zeta functions and it is a meromorphic function on the whole complex plane with poles at most in the set
$\{1,2,\ldots,m\}$. We compute the residues of $\zeta_M(z,w)$ in terms of the quasi-polynomial $q_M(n)$. Other properties of $\zeta_M(z,w)$ are given in Proposition $1.2$, $1.3$ and Corollary $1.4$, $1.5$. 

We also consider the function $\zeta_M(z) := \lim_{w\searrow 0}(\zeta_M(z,w)-H(M,0)w^{-z}).$
In Proposition $1.6$ we compute $\zeta_M(z)$ and its residues. In the second section, we apply the results obtained in the first section in the case when $S=K[x_1,\ldots,x_r]$ is the ring of
polynomials with $\deg(x_i)=a_i$, $1\leq i\leq r$. Given a graded $S$-module $M$, we compute the residues of $\zeta_M(z,w)$ and $\zeta_M(z)$
in terms of the graded Betti numbers of $M$ and the Bernoulli-Barnes polynomial associated to $(a_1,\ldots,a_r)$, see Corollary $2.2$.

In the third section, we 
consider the standard graded case and we prove that the multiplicity of $M$, is
$$e(M)=(m-1)!\lim_{w\searrow 0}\Res_{z=m}\zeta_M(z,w),$$ see Corollary $3.3$. In the fourth case, we outline the non-graded case and we
give a formula for the multiplicity of the module with respect to an ideal, see Proposition $4.1$.

\newpage
\section{Graded modules over positively graded $K$-algebras}
Let $K$ be a field and let $S$ be a positively graded $K$-algebra, that is $$S:=\bigoplus_{n\geq 0} S_n, S_0=K,$$
and $S$ is finitely generated over $K$. Assume $S=K[u_1,\ldots,u_r]$, where $u_i\in S$ are homogeneous elements of $\deg(u_i)=a_i$.
Let 
$$M=\bigoplus_{n\in\mathbb N}M_n$$ be a finitely generated graded $S$-module with the Krull dimension $m:=\dim(M)$.
The \emph{Hilbert function} of $M$ is $$H(M,-):\mathbb N\rightarrow \mathbb N,\; H(M,n):=\dim_{K}(M_n),\;n\in\mathbb N.$$
The \emph{Hilbert series} of $M$ is $$H_M(t):=\sum_{n=0}^{\infty}H(M,n)t^n \in \mathbb Z[[t]].$$
According to the Hibert-Serre's Theorem \cite[Theorem 11.1]{ati} and \cite[Exercise 4.4.11]{bh}
$$ H_M(t) = \frac{h_M(t)}{(1-t^{a_1})\cdots (1-t^{a_r})},$$
where $h_M(t)\in \mathbb Z[t]$. According to Serre's Theorem \cite[Theorem 4.4.3]{bh} and \cite[Exercise 4.4.11]{bh} there exists a 
quasi-polynomial $q_M(n)$
of degree $m-1$ with the period $D:=\lcm(a_1,\ldots,a_r)$ such that
\begin{equation}
 H(M,n)=q_M(n)=d_{M,m-1}(n)n^{m-1}+\cdots +d_{M,1}(n)n+d_{M,0}(n),\;(\forall)n\gg 0,
\end{equation}
where $d_{M,k}(n+D)=d_{M,k}(n)$ for any $n\geq 0$ and $0\leq k\leq m-1$.
We denote 
\begin{equation}
\alpha(M):=\min\{n_0\;:\:  H(M,n)=q_M(n),\;(\forall)n\geq n_0\}.
\end{equation}
Let $w>0$ be a real number. We denote
\begin{equation}
\zeta_M(z,w):=\sum_{n\geq 0}\frac{H(M,n)}{(n+w)^z},\; z\in\mathbb C,
\end{equation}
and we call the \emph{Zeta-Barnes type function} associated to $M$ and $w$. We also denote 
\begin{equation}
\theta_M(z,w):=\sum_{n=0}^{\alpha(M)-1}\frac{H(M,n)}{(n+w)^z},\; z\in\mathbb C.
\end{equation}
The function $\theta_M(z,w)$ is entire. Moreover, $M$ is Artinian if and only if $\zeta_M(z,w)=\theta_M(z,w)$.
Also, $\alpha(M)=0$ if and only if $\theta_M(z,w)=0$.

\begin{teor}
We have that {\small
$$\zeta_M(z,w)= \theta_M(z,w)+ D^{-z}\sum_{k=0}^{m-1} \sum_{j=0}^{D-1} d_{M,k}(j+\alpha(M)) \sum_{\ell=0}^k  
\binom{k}{\ell}(-w)^{\ell} D^{k-\ell}\zeta(z-k+\ell, \frac{j+\alpha(M)+w}{D}), $$}
where $\zeta(z,w)=\sum_{n=0}^{\infty}\frac{1}{(n+w)^z}$ is the \emph{Hurwitz-zeta} function. 

Moreover,  $\zeta_M(z,w)$ is a meromorphic function on $\mathbb C$ 
with the poles in the set $\{1,2,\ldots,m\}$ which are simple with residues
$$ R_M(w,k+1):=\Res_{z=k+1} \zeta_M(z,w) = \frac{1}{D} \sum_{\ell=k}^{m-1} \binom{\ell}{k}(-w)^{\ell-k}\sum_{j=0}^{D-1} d_{M,k}(j),\;0\leq k\leq m-1.$$
\end{teor}

\begin{proof}
The proof follows the line of the proof of \cite[Proposition 3.2]{lucrare}. According to $(1.1)$, $(1.2)$, $(1.3)$ and $(1.4)$, we have
\begin{equation}
\zeta_M(z,w)=\theta_M(z,w)+ \sum_{n=\alpha(M)}^{\infty}\frac{q_M(n)}{(n+w)^z}=
 \theta_M(z,w)+ \sum_{k=0}^{m-1} \sum_{n=\alpha(M)}^{\infty} \frac{d_{M,k}(n)n^k}{(n+w)^z}.
\end{equation}
For any $0\leq k\leq m-1$, we write 
\begin{equation}
n^k = (n+w-w)^k = \sum_{\ell=0}^k (-1)^{\ell} \binom{k}{\ell}(n+w)^{k-\ell}w^{\ell}.
\end{equation}
By $(1.5)$ and $(1.6)$ and the fact that $d_{M,k}(n+D)=d_{M,k}(n)$, $(\forall)n,k$, it follows that
$$\zeta_M(z,w)=\theta_M(z,w)+\sum_{k=0}^{m-1} \sum_{n=\alpha(M)}^{\infty} d_{M,k}(n) \sum_{\ell=0}^k (-1)^{\ell} 
\binom{k}{\ell}w^{\ell} \frac{1}{(n+w)^{z-k+\ell}} = \theta_M(z,w) +$$
\begin{equation}
 + \sum_{k=0}^{m-1} \sum_{j=0}^{D-1} d_{M,k}(j+\alpha(M)) \sum_{\ell=0}^k (-1)^{\ell} 
\binom{k}{\ell}w^{\ell} \sum_{t=0}^{\infty} \frac{1}{(j+tD+\alpha(M)+w)^{z-k+\ell}}.
\end{equation}
On the other hand, {\small
\begin{equation} \sum_{t=0}^{\infty} \frac{1}{(j+tD+\alpha(M)+w)^{z-k+\ell}} = 
\sum_{t=0}^{\infty} \frac{D^{-z+k-\ell}}{(t+\frac{j+\alpha(M)+w}{D})^{z-k+\ell}} = 
D^{-z+k-\ell}\zeta(z-k+\ell, \frac{j+\alpha(M)+w}{D}).
\end{equation}} 
Replacing $(1.8)$ in $(1.7)$ we get the required result. 

The last assertion is a consequence of the fact
that the Hurwitz-zeta function $\zeta(z-k,w)$ is a meromorphic function and has a simple pole at $k+1$ with the residue $1$ and, 
also, $\theta_M(z,w)$ is an entire function.
\end{proof} \pagebreak

\begin{prop}
Let $0\rightarrow U \rightarrow M \rightarrow N \rightarrow 0$ be a graded short exact sequence of $S$-modules. Then
$$\zeta_{M}(z,w) = \zeta_U(z,w) + \zeta_N(z,w).$$
\end{prop}

\begin{proof}
It follows from $H(M,n)=H(U,n)+H(N,n)$, $n\geq 0$, and $(1.3)$.
\end{proof}

\begin{prop}
For any $k\geq 0$, it holds that $\zeta_{M(-k)}(z,w)=\zeta_M(z,w+k)$. 
\end{prop}

\begin{proof}
Since $M(-k)_n=M_{n-k}$, it follows that $H(M(-k),n)=0$ for all $0\leq n < k$ and $H(M(-k),n)=H(M,n-k)$, for all $n\geq k$.
Consequently, by $(1.3)$, we get $$\zeta_{M(-k)}(z,w)=\sum_{n=0}^{\infty}\frac{H(M(-k),n)}{(n+w)^z} = \sum_{n=k}^{\infty} \frac{H(M,n-k)}{(n+w)^z} = \sum_{n=0}^{\infty}\frac{H(M,n)}{(n+k+w)^z}=\zeta_M(z,w+k).$$
\end{proof}

\begin{cor}
If $f\in S_k$ is regular on $M$, then $$\zeta_{\frac{M}{fM}}(z,w)=\zeta_M(z,w)-\zeta_M(z,w+k).$$
\end{cor}

\begin{proof}
We consider the short exact sequence $$0 \rightarrow M(-k) \stackrel{\cdot f}{\rightarrow} M \rightarrow \frac{M}{fM} \rightarrow 0.$$
The conclusion follows from Proposition $1.2$ and Proposition $1.3$.
\end{proof}

\begin{cor}
If $f_1,\ldots,f_p \in S$ is a regular sequence on $M$, consisting of homogeneous elements with $\deg(f_i)=k_i$, then
$$ \zeta_{\frac{M}{(f_1,\ldots,f_p)M}}(z) = \zeta_M(z,w) + \sum_{\ell=1}^p (-1)^{\ell} \sum_{1\leq i_1 < \cdots < i_{\ell}\leq p} \zeta_M(z,w+k_{i_1}+\ldots+k_{i_{\ell}}). $$
\end{cor}

\begin{proof}
It follows from Corollary $1.4$, using induction on $k\geq 1$.
\end{proof}

Let
\begin{equation}
 \zeta_M(z):=\lim_{w\searrow 0}(\zeta_M(z,w)-H(M,0)w^{-z}) = \sum_{n=1}^{\infty}\frac{H(M,n)}{n^z}.
\end{equation}
Note that $\zeta_M(z)$ codify all the information about the Hilbert function of $M$ with the exception of $H(M,0)$. Let
\begin{equation}
 \theta_M(z):=\sum_{n=1}^{\alpha(M)-1}\frac{H(M,n)}{n^z}.
\end{equation}
Note that $\theta_M(z)$ is an entire function. Also, if $\alpha(M)\leq 1$ then $\theta_M(z)$ is identically zero.

\begin{prop}
 We have that
$$ \zeta_M(z)= \theta_M(z) + \sum_{k=0}^{m-1} \frac{1}{D^{z-k}} \sum_{j=0}^{D-1} d_{M,k}(j+\alpha(M)) \zeta(z-k, \frac{j+\alpha(M)+1}{D}). $$
The function $\zeta_M(z)$ is meromorphic with poles at most in the set $\{1,\ldots,m\}$ which are all simple with residues
$$ R_M(k+1) := \Res_{z=k+1}\zeta_M(z) = \frac{1}{D} \sum_{j=0}^D d_{M,k}(j),\;0\leq k\leq m-1.$$
\end{prop}

\begin{proof}
The proof is similar to the proof of Theorem $1.1$, therefore we will omite it. Also, the result could be derived from the proof of \cite[Proposition 3.4(i)]{lucrare}.
\end{proof}

Let $k\geq 1$ be an integer and let 
$$M(k):=\bigoplus_{n=-k}^{\infty}M_{n+k}.$$
Given a real number $w>k$, we consider the function
\begin{equation}
\zeta_{M(k)}(z,w) := \sum_{n=-k}^{\infty}\frac{H(M,n+k)}{(n+w)^z} = \sum_{n=0}^{\infty}\frac{H(M,n)}{(n+w-k)^z} = \zeta_M(z,w-k).
\end{equation}
Let $a(S):=\deg(H_S(t))$ be the $a$-invariant of $S$. Assume $S$ is Gorenstein. Then, according to \cite[Proposition 3.6.11]{bh}, the canonical module of $S$, $\omega_S$ is isomorphic to $S(a(S))$.
Consequently, we get $\zeta_{\omega_S}(z,w) = \zeta_S(z,w-a(S))$, where $w>\max\{0,a(s)\}$. 

\begin{prop}
Let $S$ be a Cohen-Macaulay domain with the canonical module $\omega_S$. Then $S$ is Gorenstein if and only if 
$\zeta_{\omega_S}(z,w) = \zeta_S(z,w-a(S))$.
\end{prop}

\begin{proof}
Note that $\zeta_{\omega_S}(z,w) = \zeta_S(z,w-a(S))$ is equivalent to $H_{\omega_S}(t) = t^{a(S)}H_S(t)$.
Hence, according to \cite[Theorem 4.4.5(2)]{bh}, this is equivalent to $S$ is Gorenstein.
\end{proof}

\begin{obs}\emph{
Assume that $S=K[x_1,\ldots,x_r]$ is the ring of polynomials with $\deg(x_i)=a_i$, $1\leq i\leq r$. The Hilbert series of $S$ is
$$H_S(t)=\frac{1}{(1-t^{a_1})\cdots(1-t^{a_r})},$$
hence $a(S)=-(a_1+\cdots+a_r)$. It is well known that $S$ is Gorenstein, therefore $$\omega_S \cong S(a(S)) = S(-a_1-\cdots-a_r).$$
It follows that $$\zeta_{\omega_S}(z,w) = \zeta_S(z,w+a_1+\cdots + a_r),\;(\forall)w>0.$$
In the next section we will discuss the case of graded modules over $S$.}
\end{obs}

\section{Graded modules over the ring of polynomials.}

Let $\mathbf a=(a_1,\ldots,a_r)$ be a sequence of positive integers. In the following, $S=K[x_1,\ldots,x_r]$ is 
the ring of polynomials in $r$ indeterminates, with $\deg(x_i)=a_i$, $1\leq i\leq r$.
The \emph{restricted partition function} associated to $\mathbf a$ is $\pa:\mathbb N \rightarrow \mathbb N$, 
$$\pa(n):= \text{the number of integer solutions $(x_1,\ldots,x_r)$ of $\sum_{i=1}^r a_ix_i=n$ with $x_i\geq 0$} .$$
For a kindly introduction on the restricted partition function we reffer to \cite{ramirez}. One can easily see that 
$\pa(n)=H(S,n)$, $(\forall)n\geq 1$, hence 
\begin{equation}
\zeta_S(z,w)=\za(z,w) := \sum_{n=0}^{\infty} \frac{\pa(n)}{(n+w)^z}
\end{equation}
is the \emph{Zeta-Barnes function} associated to the sequence $\mathbf a$. We also have
\begin{equation}
 \zeta_S(z) = \za(z) := \lim_{w\searrow 0}(\za(z,w)-w^z)= \sum_{n=1}^{\infty} \frac{\pa(n)}{n^z}.
\end{equation}
See \cite{lucrare} for further details on the properties of the function $\za(z)$.

\begin{prop}
Let $M$ be a finitely generated graded $S$-module. Then:
\begin{enumerate}
 \item[(1)] $\zeta_M(z,w):=\sum_{i=0}^p (-1)^i \sum_{j\geq i}\beta_{ij}(M) \za(z,w+j)$,
where $\beta_{ij}(M):=\dim_K (Tor_i(M,K))_j$ are the \emph{graded Betti numbers} of $M$ and $p$ is 
the \emph{projective dimension} of $M$. 
\item[(2)] $\zeta_M(z) = \sum_{i=0}^p (-1)^i \sum_{j\geq \max\{i,1\}}\beta_{ij}(M) \za(z,j) + \beta_{00}(M)\za(z)$.
\end{enumerate}
\end{prop}

\begin{proof}
$(1)$ Let \begin{equation}                    
 \mathbf F: 0\rightarrow F_p \rightarrow \cdots \rightarrow F_1 \rightarrow F_0 \rightarrow M \rightarrow 0,
\end{equation}
be the minimal free resolution of $M$. We have that $F_i=\bigoplus_{j\geq 0}S(-j)^{\beta_{ij}}$. By $(2.1)$, Proposition $1.2$ and Proposition $1.3$, it follows that 
$$\zeta_{F_i}(z,w)=\sum_{j\geq 0}\beta_{ij}\za(z,w+j).$$ 
The result follows from Proposition $1.2$ applied several times to the exact sequence $(2.3)$.
$(2)$ By $(2.1)$, it follows that 
\begin{equation}
\lim_{w\searrow 0}\za(z,j+w)=\za(z,j),\; (\forall)j\geq 1.
\end{equation}
Using $(2.2)$, $(2.4)$ and $(1)$ we get the required result.
\end{proof}

The Bernoulli numbers $B_{\ell}$ are defined by
$$\frac{z}{e^z-1}=\sum_{\ell=0}^{\infty}B_j \frac{z^{\ell}}{\ell!},$$
$B_0=1$, $B_1=-\frac{1}{2}$, $B_2=\frac{1}{6}$, $B_4=-\frac{1}{30}$ and $B_n=0$ if $n\geq 3$ is odd. For $k>0$ we have the Faulhaber's identity
$$1^k+2^k+\cdots+n^k = \frac{1}{k+1}\sum_{\ell=0}^{k}\binom{k+1}{\ell}B_{\ell} n^{1+k-\ell}. $$
The Bernoulli-Barnes polynomials $B_{\ell}(x;a_1,\ldots,a_r)$ are defined by
$$\frac{z^r e^{xz}}{(e^{a_1z}-1)\cdots (e^{a_rz}-r)}=\sum_{\ell=0}^{\infty}B_{\ell}(x;a_1,\ldots,a_r) \frac{z^{\ell}}{\ell!}.$$
 According to formula (3.9) in Ruijsenaars \cite{rui},
\begin{equation}
\Res_{z=\ell} \za(z,w) = \frac{(-1)^{r-\ell}}{(\ell-1)!(r-\ell)!}B_{r-\ell}(w;a_1,\ldots,a_r),\;1\leq \ell \leq r.
\end{equation}
The Bernoulli-Barnes numbers are defined by
$$B_{\ell}(a_1,\ldots,a_r) := B_{\ell}(0;a_1,\ldots,a_r).$$
The Bernoulli-Barnes numbers and the Bernoulli numbers are related by
$$B_{\ell}(a_1,\ldots,a_r) = \sum_{i_1+\cdots+i_r=\ell}\binom{\ell}{i_1,\ldots,i_r}B_{i_1}\cdots B_{i_r}a_1^{i_1-1}\cdots a_r^{i_r-1}, $$
see Bayad and Beck \cite[Page 2]{babeck} for further details. According to \cite[Theorem 3.10]{lucrare},
\begin{equation}
\Res_{z=\ell}\za(z) = \frac{(-1)^{r-\ell}}{(\ell-1)!(r-\ell)!}B_{r-\ell}(a_1,\ldots,a_r),\;1\leq \ell\leq r.
\end{equation}
Note that $(2.6)$ can be deduced from $(2.5)$.

\begin{cor}
Let $M$ be a finitely generated graded $S$-module and $w>0$. Then
\begin{enumerate}
 \item[(1)] $R_M(w,\ell) = \sum_{i=0}^p \sum_{j\geq 0}\beta_{ij}(M) \frac{(-1)^{i+r-\ell}}{(\ell-1)!(r-\ell)!}B_{r-\ell}(w+j;a_1,\ldots,a_r),\;1\leq \ell \leq r $.
 \item[(2)] $R_M(\ell) = \sum_{i=0}^p \sum_{j\geq 0}\beta_{ij}(M)\frac{(-1)^{i+r-\ell}}{(\ell-1)!(r-\ell)!}B_{r-\ell}(j;a_1,\ldots,a_r),\;1\leq \ell\leq r. $.
\end{enumerate}
\end{cor}

\begin{proof}
The results follow from Proposition $2.1$ and the formulas $(2.5)$ and $(2.6)$.
\end{proof}

\begin{exm}\emph{
Let $\mathbf a = (a_1,\ldots,a_r)$ be a sequence of positive integers, $D=\lcm(a_1,\ldots,a_r)$. 
We consider the ideal $I=(x_1^{\frac{D}{a_1}},\ldots,x_r^{\frac{D}{a_r}})\subset S$. Note that $I$ is an Artinian complete intersection 
monomial ideal generated in degree $D$, w.r.t. the $\mathbf a$-grading. According to $(2.2)$ and Corollary $1.6$, we have
\begin{equation}
\zeta_{S/I}(z,w) = \theta_{S/I}(z,w) = \sum_{j=0}^r (-1)^j \binom{r}{j} \za(z,w+Dj).
\end{equation}
On the other hand, one can easily check that
$$H_{S/I}(t)=\frac{(1-t^D)^r}{(1-t^{a_1})\cdots (1-t^{a_r})} = (1+t^{a_1}+\cdots+t^{a_1(\frac{D}{a_1}-1)} )\cdots (1+t^{a_r}+\cdots+t^{a_r(\frac{D}{a_r}-1)})$$
is a reciprocal polynomial of degree $Dr-a_1-\cdots-a_r$. The coefficient of $t^n$ in $H_{S/I}(t)$ equals to 
$$\fa(n) = \# \{ (x_1,\ldots,x_r)\in\mathbb Z^r \;:\; a_1x_1+\cdots+a_rx_r=n,\; 0\leq x_1<\frac{D}{a_1}-1,\ldots, 0\leq x_r<\frac{D}{a_r}-1\}.$$
By $(2.7)$ it follows that
$$ \sum_{n=0}^{Dr-a_1-\cdots-a_r}\fa(n)(n+w)^{-z} = \sum_{j=0}^r (-1)^j \binom{r}{j} \za(z,w+Dj).$$
See R\o dseth and Sellers \cite{rodseth} for further details on the coefficients $\fa(n)$.}
\end{exm}

\begin{exm}
\emph{
Let $S=K[x_1,x_2]$ with $\deg(x_1)=2$, $\deg(x_2)=3$. Let $\mathbf a = (2,3)$. 
The polynomial $f=x_1^3-x_2^2\in S$ is homogeneous of degree $6$. Let $R=S/(f)$. $R$ has the minimal graded free resolution
\begin{equation}
 0 \rightarrow S(-6) \stackrel{\cdot f}{\rightarrow} S \rightarrow R \rightarrow 0
\end{equation}
It follows that the non-zero Betti numbers of $R$ are $\beta_{00}(R)=1$ and $\beta_{16}(R)=1$.
Let $w>0$. According to $(2.1)$ and Corollary $1.4$ (or $(2.8)$ and Proposition $2.1(1)$) we have 
$$\zeta_{R}(z,w) = \za(z,w) - \za(z,w+6) =  \sum_{n=0}^{\infty}\frac{\pa(n)}{(n+w)^z} - \sum_{n=0}^{\infty}\frac{\pa(n)}{(n+w+6)^z} = $$ 
$$ = \sum_{n=0}^5 \frac{\pa(n)}{(n+w)^z} + \sum_{n=6}^{\infty}\frac{\pa(n)-\pa(n-6)}{(n+w)^z} =  \frac{1}{w^z}+\sum_{n=2}^{\infty} \frac{1}{(n+w)^z} = \frac{1}{w^z} + \zeta(z,w+2).$$
In particular, the Hilbert series of $R$ is
$$H_R(t) =  1 + \sum_{n=2}^{\infty} t^n  = 1 +\frac{t^2}{1-t} = \frac{t^2-t+1}{1-t},$$
hence $\alpha(R)=a(R)=1$. It follows that $\theta_R(z,w)=\frac{1}{w^z}$. Also,
$$\zeta_R(z) = \lim_{w \searrow 0}(\zeta_R(z,w) - \frac{1}{w^z}) = \zeta(z,2) \text{ and } \theta_R(z)=0.$$
}
\end{exm}

\section{The standard graded case}

Let $S$ be a standard graded $K$-algebra, that is $S=\bigoplus_{n\geq 0}S_n$, $S_0=K$ and $S=K[S_1]$.
Let $M$ be a finitely generated graded $S$-module. According to the Hilbert-Serre's Theorem, it holds that
\begin{equation}
H_M(t)=\frac{h_M(t)}{(t-1)^m}, 
\end{equation}
where $h_M\in \mathbb Z[t]$, $m=\dim(M)$ and $h_M(1)\neq 0$. 
Also, there exists a polynomial $P_M(t)\in \mathbb Z[t]$ of degree $m-1$, 
such that $$H(M,n)=P_M(n),\; (\forall)n\gg 0,$$
which is called the \emph{Hilbert polynomial} of $M$.

The number $e(M):=h_M(1)$ is called the \emph{multiplicity} of the module $M$.

\begin{prop}
If $P_M(t)=d_{M,m-1}t^{m-1}+\cdots+d_{M,1}t+d_{M,0}$ is the Hilbert polynomial of $M$, then
$$\zeta_M(z,w)=\theta_M(z,w)+ \sum_{k=0}^{m-1} d_{M,k} \sum_{\ell=0}^k  
\binom{k}{\ell}(-w)^{\ell} \zeta(z-k+\ell, \alpha(M)+w) $$
is a meromorphic function on $\mathbb C$  with the poles in the set $\{1,2,\ldots,m\}$ which are simple with residues
$$ R_M(w,k+1):=\Res_{z=k+1} \zeta_M(z,w)=\sum_{\ell=k}^{m-1} \binom{\ell}{k}(-w)^{\ell-k} d_{M,\ell},\;0\leq k\leq m-1.$$
\end{prop}

\begin{proof}
It is the particular case of Theorem $1.1$ for $\mathbf a=(1,\ldots,1)$.
\end{proof}

\begin{prop}
We have that $$\zeta_M(z)=\theta_M(z)+\sum_{k=0}^{m-1} d_{M,k} \zeta(z-k+\ell, \alpha(M)+1) $$
is a meromorphic function on $\mathbb C$  with the poles in the set $\{1,2,\ldots,m\}$ which are simple with residues
$$ R_M(\ell+1):=\Res_{z=\ell+1} \zeta_M(z) = d_{M,\ell}.$$
\end{prop}

\begin{proof}
It is the particular case of Proposition $1.6$ for $\mathbf a=(1,\ldots,1)$.
\end{proof}

If $\dim M\geq 1$, then we can write
\begin{equation}
 P_M(t)=\sum_{k=0}^{m-1} (-1)^{k} e_{k}(M) \binom{t+m-1-k}{m-1-k}.
\end{equation}
According to \cite[Proposition $4.1.9$]{bh}, we have
\begin{equation}
 e_{k}(M) = \frac{h^{(k)}_M(t)}{k !},\;(\forall)0\leq k\leq m-1.
\end{equation}

\begin{cor}
 We have that $$e(M)=e_0(M)=(m-1)!d_{M,m-1} = (m-1)! R_M(m).$$
\end{cor}

\begin{proof}
 It follows from $(3.2)$, $(3.3)$ and Proposition $3.2$.
\end{proof}

The \emph{higher iterated Hilbert functions} $H_i(M,n)$, $i\in\mathbb N$, of a finitely generated $S$-module $M$ are defined recursively as
follows:
\begin{equation}
H_0(M,n):=H(M,n), \text{ and } H_i(M,n) = \sum_{j=0}^n H_{i-1}(M,n),\;i\geq 1.
\end{equation}
The functions $H_i(M,n)$ are of polynomial type of degree $m+i-1$, hence
\begin{equation}
 H_i(M,n) = P_i(M,n):= d^i_{M,m+i-1}n^{m+i-1}+\cdots+d^i_{M,1}n + d^i_{M,0},\;(\forall)n\gg 0.
\end{equation}
We define the \emph{higher Zeta-Barnes type functions} associated to $M$ as follows:
\begin{equation}
\zeta^i_M(z,w):=\sum_{n=0}^{\infty} \frac{H_i(M,n)}{(n+w)^z},\;i\geq 0.
\end{equation}
and
\begin{equation}
\zeta^i_M(z)=\lim_{w\searrow 0}(\zeta^i_M(z,w)-H(M,0)w^{-z}),\;i\geq 0.
\end{equation}
Let 
$$\alpha^i(M):=\min\{n_0\in\mathbb N \;:\; H_i(M,n)=P_i(M,n),(\forall)n\geq n_0\}.$$
We define 
$$ \theta^i_M(z,w)=\sum_{n=0}^{\alpha_i(M)-1}\frac{H_i(M,n)}{(n+w)^z} \text{ and } \theta^i_M(z)=\sum_{n=1}^{\alpha^i(M)-1}\frac{H_i(M,n)}{n^z}.$$
Similar to Proposition $2.1$ and Proposition $2.2$ we have the following result.

\begin{prop}
With the above notations:
\begin{enumerate}
 \item[(1)] $\zeta^i_M(z,w)=\theta^i_M(z,w) + \sum_{k=0}^{m+i-1} d^i_{M,k} \sum_{\ell=0}^k  
\binom{k}{\ell}(-w)^{\ell} \zeta(z-k+\ell, \alpha^i(M)+w)$
is a meromorphic function on $\mathbb C$  with the poles in the set $\{1,2,\ldots,m+i\}$ which are simple with residues
$$ R^i_M(w,k+1):=\Res_{z=k+1} \zeta_M(z,w)=\sum_{\ell=k}^{m+i-1} \binom{\ell}{k}(-w)^{\ell-k} d^i_{M,\ell},\;0\leq k\leq m+i-1.$$
\item[(2)] $\zeta^i_M(z)=\theta^i_M(z)+\sum_{k=0}^{m+i-1} d^i_{M,k} \zeta(z-k+\ell, \alpha^i(M)+1) $
is a meromorphic function on $\mathbb C$  with the poles in the set $\{1,2,\ldots,m+i\}$ which are simple with residues
$$ R^i_M(k+1):=\Res_{z=k+1} \zeta_M(z) = d^i_{M,k},\;0\leq k\leq m+i-1. $$
\end{enumerate}
\end{prop}

\begin{cor}
We have that $e(M)=m!R^1_M(m+1)$.
\end{cor}

\begin{proof}
According to \cite[Remark 4.1.6]{bh}, $H_1(M,n)= d^1_{M,m}n^m + \cdots + d^1_{M,1}n+d^1_{M,0}$, $(\forall)n\gg 0$,
and $e(M)=m!d^1_{M,m}$. Now, apply Proposition $3.4(2)$.
\end{proof}

\begin{obs}\emph{
Let $S=K[x_1,\ldots,x_r]$ and $I\subset S$ a graded ideal.
We say that $S/I$ has a \emph{pure resolution} of type $(d_1,\ldots,d_p)$ if its minimal resolution is
$$0 \rightarrow S(-d_p)^{\beta_p} \rightarrow \cdots  \rightarrow S(-d_1)^{\beta_1} \rightarrow S \rightarrow S/I \rightarrow 0, $$
where $p$ is the projective dimension of $S/I$, $d_1<d_2<\cdots<d_p$ and $\beta_i=\sum_{j\geq 0}\beta_{ij}(S/I)$, $1\leq i\leq p$, are the Betti numbers of $S/I$.
According to Corollary $3.3$, $e(S/I)=R_{S/I}(m)$, where $m=\dim(S/I)$. On the other hand, according to Corollary $2.2(2)$, we have
\begin{equation}
R_{S/I}(m)=\sum_{i=0}^p \beta_i \frac{(-1)^{r-m}}{(m-1)!(r-m)!}B_{r-m}(d_i;1,1,\ldots,1).
\end{equation}
Suppose $S/I$ is Cohen-Macaulay and has a pure resolution of type $(d_1,\ldots,d_p)$. According to \cite[Theorem 4.1.15]{bh},
\begin{equation}
\beta_i = (-1)^{i+1}\prod_{j\neq i}\frac{d_j}{d_j-d_i} \text{ and }  e(S/I)=\frac{d_1d_2\cdots d_p}{p!}.
\end{equation}
The Ausl\"ander-Buchsbaum formula \cite[Theorem 1.3.3]{bh} implies $p=r-m$, hence $(3.8)$ and $(3.9)$ give the identity:
$$ \sum_{i=0}^p (-1)^{i+1}\prod_{j\neq i}\frac{d_j}{d_j-d_i} B_{p}(d_i;1,1,\ldots,1) = (m-1)!(-1)^p d_1d_2\cdots d_p.$$}
\end{obs}

\section{The non-graded case}

Let $(S,\mathfrak m,K)$ be a Noetherian local ring, where $\mathfrak m$ is the maximal ideal of $S$ and $K=S/\mathfrak m$ is the residue field.
Let $M$ be a finitely generated $S$-module, with $m=\dim(M)$, and let $I\subset S$ be an ideal such that $\mathfrak m^n M \subset IM$ for some $n\geq 1$.
The associated graded ring is
$$\gr_I(S) = \bigoplus_{n\geq 0}\frac{I^n}{I^{n+1}} = \frac{S}{I}\oplus \frac{I}{I^2} \oplus \cdots.$$
The associated graded module of $M$, with respect to $I$, is 
$$\gr_{I}(M):=\bigoplus_{n\geq 0}\frac{I^nM}{I^{n+1}M},$$
which has a structure of a $\gr_I(S)$-module. According to \cite[Theorem 4.5.6]{bh}, it holds that 
$$\dim(\gr_I(M))=\dim(M)=m.$$
The \emph{Hilbert-Samuel function} of $M$, w.r.t. $I$, is 
$$\chi_M(n):=H_1(\gr_{I}(M),n) = \sum_{i=0}^n H(\gr_{I}(M),i) = \dim_K \frac{M}{I^{n+1}M},\;(\forall)n\geq 0.$$
The \emph{multiplicity} of $M$ with respect to $I$ is $e(M,I):=e(\gr_{I}(M))$.
For $n\gg 0$, according to \cite[Remark 4.1.6]{bh}, we have that 
\begin{equation}
 \chi_M(n) = \frac{e(M,I)}{m!}n^d + \text{ terms in lower powers of } n.
\end{equation}
We consider the functions
\begin{equation}
\zeta^i_{M,I}(z,w):=\zeta^i_{\gr_{I}(M)}(z,w) \text{ and } \zeta^i_{M,I}(z):=\zeta^i_{\gr_{I}(M)}(z),\;i\geq 0.
\end{equation}

\begin{prop}
 It holds that $$e(M,I) = m! \Res_{z=m+1}\zeta^1_{M,I}(z).$$
\end{prop}

\begin{proof}
 This follows from $(4.1)$, $(4.2)$ and Corollary $3.5$.
\end{proof}

{}

\vspace{2mm} \noindent {\footnotesize
\begin{minipage}[b]{15cm}
Mircea Cimpoea\c s, Simion Stoilow Institute of Mathematics, Research unit 5, P.O.Box 1-764,\\
Bucharest 014700, Romania, E-mail: mircea.cimpoeas@imar.ro
\end{minipage}}

\end{document}